\documentclass[11pt]{amsart}

\usepackage{amsmath,amssymb,amsthm}
\usepackage{diagbox}
\usepackage{multirow}
\usepackage{array}
\usepackage{enumitem}
\usepackage{color}
\usepackage[colorlinks]{hyperref}
\usepackage[capitalize]{cleveref}
\usepackage{bm}
\usepackage{amsfonts,amssymb}
\usepackage{dsfont}
\usepackage{amscd}
\usepackage{extarrows}
\usepackage{amsmath}
\usepackage{mathrsfs}
\usepackage{amscd}
\usepackage[top=30truemm,bottom=30truemm,left=25truemm,right=25truemm]{geometry}

\newtheorem{theorem}{Theorem}[section]
\newtheorem{definition}{Definition}[section]
\newtheorem{lemma}{Lemma}[section]

\newtheorem{proposition}{Proposition}[section]
\newtheorem{remark}{Remark}[section]

\newtheorem{conj}{Conjecture}[section]

\begin{document}

	\title{A Newton-Okounkov Body Viewpoint on the SOS Conjecture}

	\author[Z. Wang]{Zhiwei Wang}
	\address{Zhiwei Wang: Laboratory of Mathematics and Complex Systems (Ministry of Education)\\ School of Mathematical Sciences\\ Beijing Normal University\\ Beijing 100875\\ P. R. China}
	\email{zhiwei@bnu.edu.cn}
	
	\author[C. Yue]{Chenlong Yue}
	\address{Chenlong Yue: School of Mathematical Sciences\\ Beijing Normal University\\ Beijing 100875\\ P. R. China}
	\email{aystl271828@163.com}
	
	\author[X. Zhou]{Xiangyu Zhou}
	\address{Xiangyu Zhou: Institute of Mathematics\\Academy of Mathematics and Systems Science\\and Hua Loo-Keng Key
		Laboratory of Mathematics\\Chinese Academy of
		Sciences\\Beijing\\100190\\P. R. China}
	\address{School of
		Mathematical Sciences, University of Chinese Academy of Sciences,
		Beijing 100049, P. R. China}
	\email{xyzhou@math.ac.cn}
	
	\maketitle 
	
	\begin{abstract}
		Let $z\in \mathbb C^n$ be the complex coordinates on $\mathbb C^n$, and $A(z,\bar z)$ be a real-valued Hermitian polynomial. The famous Ebenfelt's SOS conjecture asks for the minimum rank of $A(z,\bar z)\|z\|^2$ under the restriction that $A(z,\bar z)\|z\|^2$ is an SOS. Assume that $A(z,\bar z)$ is   bihomogeneous. In the present note,  we establish a  connection between Ebenfelt's (Weak) SOS Conjecture and the theory of Newton-Okounkov bodies. By reformulating the conjecture in terms of lattice semigroups and their associated Newton-Okounkov convex bodies, we transform the problem of finding the minimal rank of a prolonged sum-of-squares polynomial into an extremal problem in convex geometry.  In particular, we prove that this minimal rank is attained at the extreme points of a specific Newton-Okounkov body. Furthermore, if $A(z,\bar z)$ is moreover diagonal, we demonstrate that  the relevant extreme points are finitely many rational points, thereby reducing the verification of the conjecture to a computationally tractable problem. This work provides a new tool for attacking the SOS Conjecture.
	\end{abstract}
	
	\tableofcontents
	\section{Introduction}
	Let $z \in \mathbb{C}^n$, $A:=A(z,\bar{z}) \in \mathbb{C}[z_1,\cdots,z_n,\bar{z}_1,\cdots,\bar{z}_n]$, and $\|z\|$ be the usual Euclidean norm. Denote ${\rm{SOS}}_n$ as the set of real polynomials that can be expressed as sums of squares of holomorphic polynomials; that is, $A(z,\bar{z}) \in {\rm{SOS}}_n$ if and only if there exist holomorphic polynomials $f_1,\cdots,f_R$ such that $A(z,\bar{z}) = \sum_{i=1}^R |f_i|^2$. If $f_1,\cdots,f_R$ are linearly independent, then $R$ (the minimal number of such polynomials for which the equality holds) is called the rank of $A(z,\bar{z})$.
	
	For any $A(z,\bar{z}) \in \mathbb{C}[z,\bar{z}]$, choose a sufficiently large integer $d$ and arrange the monomials in $\mathbb{C}^n$ of degree at most $d$ in lexicographical order as
	\[
	\mathfrak{Z} = (1,z_1,\cdots,z_n,z_1^2,\cdots,z_1z_n,\cdots,z_n^d)
	\]
	in a row. Then there exists a Hermitian matrix $H$ such that $A(z,\bar{z}) = \mathfrak{Z}H\mathfrak{Z}^*$, where $\mathfrak{Z}^*$ denotes the conjugate transpose of $\mathfrak{Z}$. In fact, $A(z,\bar{z}) \in {\rm{SOS}}_n$ if and only if $H$ is positive semidefinite. By the diagonalization algorithm for Hermitian matrices, the rank of the polynomial $A$ coincides with the rank of the matrix $H$, which is independent of the choice of $d$. 
	Throughout this paper, when we refer to the eigenvalues and trace of a polynomial in the subsequent text, we essentially mean those of the corresponding matrix defined in the above manner.
	
	\begin{definition}[Prolongation map]
		The prolongation map in $\mathbb{C}[z_1,\cdots,z_n,\bar{z}_1,\cdots,\bar{z}_n]$ is defined as
		\begin{equation}
			J_n : \mathbb{C}[z,\bar{z}] \to \mathbb{C}[z,\bar{z}], \quad A(z,\bar{z}) \mapsto A(z,\bar{z})\|z\|^2.
		\end{equation}
	\end{definition}
	
	The prolongation map is a linear injection. With respect to the aforementioned polynomial basis, the restriction of the prolongation map to a finite-dimensional subspace of $\mathbb{C}[z,\bar{z}]$ can be represented by a matrix, which is explicit and well-defined. Surprisingly, when the positive semidefiniteness constraint is imposed, the rank of the image of the prolongation map cannot take all positive values; instead, there exist regular gaps between certain intervals of ranks. Inspired by the Huang-Ji-Yin Gap Conjecture \cite{HJY}, Ebenfelt proposed the following conjecture \cite{E1}:
	
	\begin{conj}[SOS Conjecture]
		Let $n \geq 2$. If $A(z,\bar{z})\|z\|^2 \in {\rm{SOS}}_n$, then the rank $R$ of $A\|z\|^2$ either satisfies
		\[
		R \geq (\kappa_0+1)n - \frac{(\kappa_0+1)\kappa_0}{2} - 1,
		\]
		where $\kappa_0$ is the largest integer such that $\frac{\kappa(\kappa+1)}{2} < n$, or there exists $\kappa \in \{0,1,2,\cdots,\kappa_0\}$ such that
		\[
		\kappa n - \frac{\kappa(\kappa-1)}{2} \leq R \leq \kappa n.
		\]
	\end{conj}
	
	A brilliant lemma by Huang \cite{H2} shows that $R=0$ or $R \geq n$, so the conjecture holds for $n=2$ by Huang's lemma. For $n \geq 3$, partial related results have been obtained \cite{B1}\cite{G1}\cite{WYZ}. In particular, when $A \in {\rm{SOS}}_n$, Dusty and Halfpap had already studied and proved this part from an algebraic perspective long before the conjecture was proposed \cite{G3}; thus, the remaining case of SOS conjecture is the following:
	\begin{conj}[(Weak) SOS Conjecture]\label{conj: weak sos} 
		Let $n \geq 2$. If $A(z,\bar{z}) \notin {\rm{SOS}}_n$ but $A(z,\bar{z})\|z\|^2 \in {\rm{SOS}}_n$, then the rank of $A(z,\bar{z})\|z\|^2$ satisfies
		\begin{align}
			{\rm{Rank}} \Big( A(z,\bar{z})\|z\|^2 \Big) \geq n(\kappa_0+1) - \frac{(\kappa_0+1)\kappa_0}{2} - 1,	\label{weak lower bound}
		\end{align}
		where $\kappa_0$ is the largest integer such that $\kappa(\kappa+1)/2 < n$.
	\end{conj}
	
	
	As pointed out by Ebenfelt, one of the main difficulties in  Conjecture \ref{conj: weak sos} comes from the fact that it seems hard to characterize when $A(z,\bar z)\|z\|^2$ is in fact an SOS.
	
	This paper focuses on the case where $A$ is bihomogeneous.
	
	Let $\mathbb{C}_{(d,d)}[z,\bar{z}]$ denote the linear space consisting of $(d,d)$-bihomogeneous real polynomials in $\mathbb{C}[z_1,\cdots,z_n,\bar{z}_1,\cdots,\bar{z}_n]$. Via the aforementioned basis representation, it is isomorphic to the real vector space of all $\binom{n+d-1}{d}$-dimensional Hermitian matrices:
	\[
	{\rm{Herm}}\left( \binom{n+d-1}{d}, \mathbb{C} \right) \cong \mathbb{R}^{\binom{n+d-1}{d}^2}.
	\]
	
	The (Weak) SOS Conjecture in the homogeneous case amounts to conjecturing that the number $R_{n,d}$, defined as
	\begin{align}
		\label{infimum rank}
		R_{n,d} := \inf\left\{ {\rm{Rank}}\Big(A(z,\bar{z})\|z\|^2 \Big) \,\bigg|\, A \in \mathbb{C}_{(d,d)}[z,\bar{z}],\ A \notin {\rm{SOS}}_n,\ J_n(A) \in {\rm{SOS}}_n \right\},
	\end{align}
	is greater than the integer on the right-hand side of equation (\ref{weak lower bound}), which depends only on $n$.
	
	Restricting the prolongation map $J_n$ to $\mathbb{C}_{(d,d)}[z,\bar{z}]$, the following theorem shows that the infimum is attained within the set of extreme points of a compact convex set.
	
	\begin{theorem}\label{thm:main-general}
		Let $n, d \geq 2$ be integers. Consider the following lattice semigroups of Hermitian matrices:
		\begin{itemize}
			\item $S_1 := \left\{ H \in \operatorname{Herm}_{\binom{n+d-1}{d}}(\mathbb{Z}[i]) \mid H \geq 0 \right\}$, corresponding to sum-of-squares polynomials.
			\item $S_2 := \left\{ H \in \operatorname{Herm}_{\binom{n+d-1}{d}}(\mathbb{Z}[i]) \mid J_n(H) \geq 0 \right\}$, corresponding to polynomials whose prolongation is a sum of squares.
		\end{itemize}
		Let $\Delta_1, \Delta_2$ be the Newton-Okounkov bodies associated with the strongly admissible pairs $(S_1, M)$ and $(S_2, M)$, respectively, where $M = \operatorname{Tr}^{-1}(\mathbb{R}_{\geq 0})$. Denote by $E(\Delta_i)$ the set of extreme points of $\Delta_i$.
		
		Then, the minimal rank \eqref{infimum rank} is attained within the set of extreme points of $\Delta_2$ that lie outside the extreme points of $\Delta_1$, i.e.,
		\[
		R_{n,d} =  \inf_{X \in E(\Delta_2) \setminus E(\Delta_1)} {\rm{Rank}}(J_n(X)).
		\]
	\end{theorem}
	Further restricting $J_n$ to the subspace $\mathbb{C}_{(d,d)}^{Diag}[z,\bar{z}]$ consisting of diagonal polynomials in $\mathbb{C}_{(d,d)}[z,\bar{z}]$, we have
	\begin{theorem}[Diagonal Case]\label{thm:main-diagonal}
		Let $n, d \geq 2$ be integers. Define the minimal rank in the diagonal case as
		\[
		R_{n,d}^{\mathrm{Diag}} := \inf \left\{ \operatorname{Rank}\left(A(z,\bar{z})\|z\|^2\right) \ \middle| \ A \in \mathbb{C}_{(d,d)}^{\mathrm{Diag}}[z,\bar{z}],\ A \notin \mathrm{SOS}_n,\ A\|z\|^2 \in \mathrm{SOS}_n \right\}.
		\]
		Consider the canonical identification $\mathbb{C}_{(d,d)}^{\mathrm{Diag}}[z,\bar{z}] \cong \mathbb{R}^{\binom{n+d-1}{d}}$ and the prolongation matrix $J_{n,d}$ from \eqref{equ: mac repr Jnd}. Define the lattice semigroups:
		\begin{itemize}
			\item $S_1^{\mathrm{Diag}} := \left\{ H \in \mathbb{Z}^{\binom{n+d-1}{d}} \mid H \geq 0 \right\}$, corresponding to diagonal sum-of-squares polynomials.
			\item $S_2^{\mathrm{Diag}} := \left\{ H \in \mathbb{Z}^{\binom{n+d-1}{d}} \mid J_{n,d} H \geq 0 \right\}$, corresponding to diagonal polynomials whose prolongation is a sum of squares.
		\end{itemize}
		Let $\Delta_1^{\mathrm{Diag}}, \Delta_2^{\mathrm{Diag}}$ be the Newton-Okounkov bodies associated with $(S_1^{\mathrm{Diag}}, M^{Diag})$ and $(S_2^{\mathrm{Diag}}, M^{Diag})$, where $M^{Diag} = {\rm{Tr}}^{-1}(\mathbb{R}_{\geq 0}) \cap \mathbb{C}_{(d,d)}^{Diag}[z,\bar{z}] $.
		
		Then, the minimal rank $R_{n,d}^{\mathrm{Diag}}$ is attained at a finite set of rational points:
		\[
		R_{n,d}^{\mathrm{Diag}} =\min_{X \in E(\Delta_2^{Diag}) \setminus E(\Delta_1^{Diag})} {\rm{Rank}}(J_n(X)).
		\]
	\end{theorem}
\begin{remark}
The results presented in this paper are restricted to the bihomogeneous setting. Observe that 
the quantity \(R_{n,d}^{\mathrm{Diag}}\)  is non-increasing in the degree \(d\).
Consequently, if the limit 
\[
\lim_{d\to\infty} R_{n,d}^{\mathrm{Diag}}
\]
exceeds the lower bound appearing in the (Weak) SOS Conjecture (2), a standard degree-by-degree argument would imply the validity of the conjecture for all diagonal polynomials (not neccessarily bihomogeneous).
However, for non-diagonal  polynomials, a suitable graded structure that relates the minimal rank for different degrees is not yet available.
Hence the method employed in the diagonal case does not extend directly, and the non-diagonal case remains open.
\end{remark}

	The above theorems significantly narrow down the range of polynomials for which the rank may attain its minimum value; the latter theorem further restricts this range to a finite set, making computer-aided proofs feasible. The proofs of the theorems only rely on a small amount of convex analysis and properties of the prolongation map, and the linearity of the prolongation map still has substantial potential for further exploration. For instance, if we take the norm to be indefinite: $\|z\|^2 = \sum_{i=1}^r |z_i|^2 - \sum_{j=1}^t |z_{r+j}|^2$ (where $r + t \leq n$ and $t < r$), the analogous theorems remain valid. 

	\subsection*{Acknowledgements}
	The first author would like to thank Professor Xiaojun Huang, Wanke Yin,  Sui-Chung Ng and Yun Gao for their interest in our work and their invaluable discussions on the SOS conjecture. This research is supported by the National Key R \& D Program of China (Grant No. 2021YFA1002600 and No. 2021YFA1003100). Z. Wang and X. Zhou are partially supported by grants from the National Natural Science Foundation of China (NSFC) (No. 12571085) and (No. 12288201) respectively. Z. Wang is also supported by the Fundamental Research Funds for the Central Universities.

	\section{Newton-Okounkov body and extreme point}
	
	We first recall basic definitions from convex analysis \cite{R1}.
	
	\begin{definition}A \textbf{convex cone} is  a non-empty convex set which is closed under non-negative scalar multiplication, i.e. $\lambda x\in K$ when $x\in K$ and $\lambda\geq 0$. 
	\end{definition}
	\begin{definition}
		The \textbf{ridge} of a closed convex cone is the maximal subspace contained within the cone. A closed convex cone is said to be \textbf{strictly convex} if its ridge is exactly the origin.
	\end{definition}
	
	\begin{definition}
		Let $C\subset \mathbb R^N$ be a convex set. A point $x\in C$ is said to be \textbf{relative interior point}, if there is a ball centered at $x$ such that its intersection with the affine hull of $C$ is entirely contained in $C$. The set of all relative interior point of $C$ is called the relative interior of $C$.
	\end{definition}
	An open interval in $\mathbb{R}^N$ is defined as the strict convex combination of two distinct points:
	\[
	(a, b) := \{\lambda a + (1-\lambda)b \mid a,b \in \mathbb R^N, a \neq b,\ 0 < \lambda < 1\}.
	\]
	
	\begin{definition}
		For any set $X \subset \mathbb{R}^N$, a point $p\in X$ is called an \textbf{extreme point} of $X$ if there is no open interval $(a,b) \subset X$ containing $p$. Denote by $E(X)$  the set of all extreme points of $X$.
	\end{definition}
	
	
	The following lemma shows that removing a closed convex set from any set does not generate new extreme points:
	
	\begin{lemma}
		\label{extreme point outside}
		Let $X_1 \subset X_2$, where $X_1$ is a closed convex set in $\mathbb{R}^N$. Then
		\[
		E(X_2 \setminus X_1) = E(X_2) \setminus E(X_1).
		\]
	\end{lemma}
	
	\begin{proof}
		We prove the two inclusions separetely.
		\begin{itemize}
			\item $E(X_2) \setminus E(X_1) \subset E(X_2 \setminus X_1) $:
			let $p \in E(X_2 ) \setminus E(X_1)$. Then $p \in X_2, p \notin X_1$, so $p \in X_2 \setminus X_1$. If $p$ were not extreme in $X_2\setminus X_1,$ there would exist distinct $a,b \in X_2\setminus X_1$ and $\lambda \in (0,1)$ with $p=\lambda a + (1-\lambda)b.$ Since $a,b \in X_2$, this contradicts $p \in E(X_2)$.
			\item $E(X_2 \setminus X_1) \subset E(X_2) \setminus E(X_1)$: let  $p \in E(X_2 \setminus X_1)$. Then $p \notin X_1$. Suppose  $p \notin E(X_2)$,  there is an open interval $(a, b) \subset X_2$ containing $p$. Since $X_1$ is closed and convex, $(a, b) \setminus X_1 \subset X_2 \setminus X_1$ is also an open set which contains an interval containing $p$, this contradicts $p \in E(X_2 \setminus X_1)$ . Therefore, if $p \in E(X_2 \setminus X_1)$, then $p \in E(X_2)$, $p \in E(X_2) \setminus E(X_1)$ .
		\end{itemize}
	\end{proof}
	
	Recall that a set $X \subset \mathbb{R}^N$ is said to be \textbf{regularly closed} if $X = \overline{\text{int}(X)}$ (the closure of its interior).

	\begin{lemma}
		\label{rational cap}
		Let $X_1 \subset X_2$ be closed sets in $\mathbb{R}^N$, with $X_2$ regularly closed. If
		\[
		(X_2 \setminus X_1) \cap \mathbb{Q}^N = \emptyset,
		\]
		then $X_1 = X_2$.
	\end{lemma}
	
	\begin{proof}
		The proof relies on the denseness of rational numbers. We only need to show $\text{int}(X_2) \subset X_1$; then, by the regularity of $X_2$ and closedness of $X_1$, $X_2 = \overline{\text{int}(X_2)} \subset X_1$.
		
		Suppose there exists $x \in \text{int}(X_2)$ such that $x \notin X_1$. Since $X_1$ is closed, there exists a small open ball $B \subset X_2 \setminus X_1$. This open ball must contain a rational vector, which is a contradiction.
	\end{proof}

	\begin{lemma}
		\label{integer cap}
		Let $X_1 \subset X_2$ be closed convex cones in $\mathbb{R}^N$, with $\text{int}(X_2) \neq \emptyset$. If
		\[
		(X_2 \setminus X_1) \cap \mathbb{Z}^N = \emptyset,
		\]
		then $X_1 = X_2$.
	\end{lemma}
	
	\begin{proof}
		By the homogeneity of cones and Lemma \ref{rational cap}, we only need to show that $X_2$ is regularly closed.  By \cite[Theorem 6.3]{R1},  any closed convex set is the closure of its relative interior. Since $\text{int}(X_2) \neq \emptyset$, it cannot be contained by any hyperplane, so the affine hull of $X_2$ is $\mathbb{R}^N$, thus its relative interior coincides with its interior.
	\end{proof}
	
	Next, we introduce the construction of the  Newton-Okounkov body, which is refered to \cite{K1}.
	
	
	Given $N > 0$, let $L$ be a linear subspace of $\mathbb{R}^N$. A half-space $M \subset L$ is a closed set whose boundary $\partial M$ is an oriented hyperplane in $L$; the inner normal side of $\partial M$ consists of all interior points of $M$. A half-space $M \subset \mathbb{R}^N$ with $0\in \partial M$  can be expressed as the solution set of an inequality:
	\[
	\langle x, \nu \rangle \geq 0, \quad x\in L,\nu \in \mathbb{R}^N,
	\]
	where $\langle \cdot, \cdot \rangle$ denotes the Euclidean inner product, and $\nu$ is the inner normal vector of $\partial M$. A half-space $M \subset L$ is said to be \textbf{rational} if both $L$ and $\partial M$ can be spanned by \textbf{rational vectors}, i.e., vectors with rational coordinates. When $\nu$ is a rational vector, the half-space defined by the inequality is rational, and rational half-spaces are invariant under rational transformations.
	
	\begin{remark}
		Regarding the definition of a half-space, standard convex analysis textbooks \cite{R1} allow $M$ to be translated by a distance in $L$ away from the origin; we will use this definition in \S \ref{sect: diagonal}. 
	\end{remark}
	
	
	Let $S$ be a semigroup of the lattice group $\mathbb{Z}^N \subset \mathbb{R}^N$. It can be associated with the following objects:
	\begin{itemize}
		\item The subspace $L(S) \subset \mathbb{R}^N$ generated by $S$. By definition, $L(S)$ is spanned by integer vectors, so the rank of the subgroup $L(S)\cap \mathbb{Z}^N$ equals $\dim L(S)$.
		\item The closed convex cone ${\rm{Con}}(S) \subset L(S)$, generated by $S$: the closure of all linear combinations of the form $\sum_i \lambda_i a_i$ where $a_i \in S$ and $\lambda_i \geq 0$.
		\item The group  $G(S) \subset L(S)$ generated by all elements of $S$: $G(S)$ consists of all linear combinations of the form $\sum_i k_i a_i$ where $a_i \in S$ and $k_i \in \mathbb{Z}$.
	\end{itemize}

	For  a rational  half-space $M \subset L$, we can associate the subgroups $\partial M_{\mathbb{Z}} = \partial M \cap \mathbb{Z}^N$ and $L_{\mathbb{Z}} = L \cap \mathbb{Z}^N$. Since all subgroups of $\mathbb{Z}^N$ are finitely generated and torsion-free, the quotient group $L_{\mathbb{Z}} / \partial M_{\mathbb{Z}}$ is a free abelian group of rank one. Hence, we can choose an element $e \in L_{\mathbb{Z}}$ whose image generates the quotient, and extend it together with a basis of $\partial M_{\mathbb{Z}}$ to form a basis of $L_{\mathbb{Z}}$. 
	
	Thus, there exists a unique linear map $\pi_M : L \to \mathbb{R}$ such that $\ker(\pi_M) = \partial M$, $\pi_M(L_{\mathbb{Z}}) = \mathbb{Z}$, and $\pi_M(M \cap \mathbb{Z}^N) = \mathbb{Z}_{\geq 0}$ (the non-negative integers). In particular, $\pi_M$ induces an isomorphism from $L_{\mathbb{Z}} / \partial M_{\mathbb{Z}}$ onto $\mathbb{Z}$.

	\begin{definition}     
		\label{Newton-Okounkov body defi}
		Let \(S\) be a semigroup in \(\mathbb{Z}^{N}\), and \(M\) be a rational half-space in \(L(S)\) such that \(S \subset M\). The pair \((S, M)\) is called a \textbf{strongly admissible pair} if the closed convex cone \(\mathrm{Con}(S)\) is strictly convex and satisfies \(\mathrm{Con}(S) \cap \partial M = \{0\}\). Its \textbf{Newton-Okounkov body} is defined as the compact convex set
		\[
		\Delta(S, M) = \mathrm{Con}(S) \cap \pi_{M}^{-1}(\mathrm{ind}(S, M)),
		\]
		where \(\mathrm{ind}(S, M)\) is the \textbf{index} of the subgroup \(\pi_M(G(S))\) in \(\mathbb{Z}\).
	\end{definition}

	\section{Newton-Okounkov body associatied with SOS conjecture}\label{sec: general case}
	We now apply the framework of Newton-Okounkov body to (Weak) SOS conjecture. 	For all $n,d \geq 2$, let
	\[
	M = {\rm{Tr}}^{-1}(\mathbb{R}_{\geq 0}),
	\]
	\[
	S_1 = \left\{ H \in {\rm{Herm}}_{\binom{n+d-1}{d}}(\mathbb{Z}[i]) \mid H \geq 0 \right\}, \quad S_2 = \left\{ H \in {\rm{Herm}}_{\binom{n+d-1}{d}}(\mathbb{Z}[i]) \mid J_n(H) \geq 0 \right\}.
	\]
	The notation ${\rm{Herm}}_{m}(\mathbb{Z}[i])$  , where $m=\binom{n+d-1}{d},$ denotes the set of all $m \times m$  Hermitian matrices with entries in the ring of Gaussian integers $\mathbb{Z}[i]$, which is isomorphic to $\mathbb Z^{m^2}$, and $\rm Tr$ denotes the trace operator on the space of Hermitian matrices $\rm{Herm}_m(\mathbb C)$.
	
	It is clear that ${\rm{Con}}(S_2) = \left\{ H \in {\rm{Herm}}_{\binom{n+d-1}{d}}(\mathbb{C}) \mid J_n(H) \geq 0 \right\}$. Actually, the previous set is contained in the latter one and the latter has non-empty interior, By Lemma \ref{integer cap}, the reverse inclusion also holds. 
	
	Similary, one can see that ${\rm{Con}}(S_1) = \left\{ H \in {\rm{Herm}}_{\binom{n+d-1}{d}}(\mathbb{C}) \mid H \geq 0 \right\}$.

	The following lemma, which describes how the prolongation map scales the trace, is fundamental. It not only facilitates the analysis of the Newton-Okounkov bodies but also imposes a necessary positivity condition on the trace.
	
	\begin{lemma}
		\label{trace n lemma}
		The prolongation map
		\[
		J_n : \mathbb{C}[z, \bar{z}] \to \mathbb{C}[z, \bar{z}], \quad A:=A(z, \bar{z}) \mapsto A(z, \bar{z})\|z\|^2=:A\|z\|^2
		\]
		scales the trace of the polynomial by a factor of $n$:
		\begin{equation}
			{\rm{Tr}}(J_n(A)) = n \cdot {\rm{Tr}}(A).
		\end{equation}
	\end{lemma}
	
	\begin{proof}
		The proof is straightforward. Choose a sufficiently large degree $d$. The trace of the matrix $H_A$ corresponding to $A$ under the lexicographical order equals the sum of its diagonal entries, i.e., the sum of the coefficients of all terms of the form $|z^\alpha|^2$ with $|\alpha| \leq d$. Multiplying $|z^\alpha|^2$ by
		\[
		\|z\|^2 = |z_1|^2 + \cdots + |z_n|^2
		\]
		yields $n$ monomials with the same coefficient but different indices. Moreover, the terms $|z^\beta|^2$ (with $|\beta| \leq d+1$) of $A\|z\|^2$ can only be generated in this way. Thus, the trace of the latter is $n$ times that of the former. 
	\end{proof}
	
	\begin{remark}
		Similarly, for all $r, s \geq 0$ with $r + s \leq n$, one can prove that
		\[
		{\rm{Tr}}\left(A\left(|z_1|^2 + \cdots + |z_r|^2 - |z_{r+1}|^2 - \cdots - |z_{r+s}|^2\right)\right) = (r - s){\rm{Tr}}(A).
		\]
	\end{remark}
	\begin{proposition}
		$(S_1, M)$ and $(S_2, M)$ are strongly admissible pairs.
	\end{proposition}
	\begin{proof}
		Since the proof is similar, we only prove it for $(S_2, M)$.
		It is easy to show that ${\rm{Con}}(S_2) = \left\{ H \in {\rm{Herm}}_{\binom{n+d-1}{d}}(\mathbb{C}) \mid J_n(H) \geq 0 \right\}$ is strictly convex. In fact, since the trace is a continuous function, $\partial {\rm{Tr}}^{-1}(\mathbb{R}_{\geq 0}) = {\rm{Tr}}^{-1}(0)$. Let $H_0 \in {\rm{Con}}(S_2) \cap {\rm{Tr}}^{-1}(0)$. By Lemma \ref{trace n lemma}, $S_2 \subset M,J_n(H_0) \geq 0$ and ${\rm{Tr}}(J_n(H_0)) = 0$, so $J_n(H_0) = 0$, since the prolongation map is injective, thus $H_0 = 0$, meaning the intersection contains only the origin, that is to say, $(S_2, M)$ is a strongly admissible pair.
	\end{proof}
	
	\noindent\textbf{Proof of Theorem \ref{thm:main-general}.}
	The proof of theorem \ref{thm:main-general} proceeds in three steps: 
	\begin{itemize}
		\item Step 1: Express the minimal rank \(R_{n,d}\) in terms of the cones;
		\item Step 2: Use a homogeneity argument to restrict the search to the Newton-Okounkov bodies;
		\item Step 3: Apply a convexity argument to further restrict the search to the extreme points.
	\end{itemize}
	\noindent \textbf{Step 1.} We first \textbf{claim} that  $\Delta(S_1, M) \subset \Delta(S_2, M)$.
	
	It is easy to verify that the morphism associated with the half-space $M$ is
	\[
	\pi_M : {\rm{Herm}}_{\binom{n+d-1}{d}}(\mathbb{C}) \to \mathbb{R}, \quad H \mapsto {\rm{Tr}}(H).
	\]
	Furthermore, since $|z_1|^{2d} \in S_1 \subset S_2$, we have $1 \in \pi_M(G(S_1)) \subset \pi_M(G(S_2))$. That is,
	\[
	{\rm{ind}}(S_1, M) = {\rm{ind}}(S_2, M) = 1.
	\]
	By Definition \ref{Newton-Okounkov body defi}, we can explicitly write these two Newton-Okounkov bodies as
	\[
	\Delta(S_1, M) = \left\{ H \in {\rm{Herm}}_{\binom{n+d-1}{d}}(\mathbb{C}) \mid H \geq 0 \right\} \cap {\rm{Tr}}^{-1}(1),
	\]
	\[
	\Delta(S_2, M) = \left\{ H \in {\rm{Herm}}_{\binom{n+d-1}{d}}(\mathbb{C}) \mid J_n(H) \geq 0 \right\} \cap {\rm{Tr}}^{-1}(1).
	\]
	It is obvious that $S_1 \subset S_2$, it follows from above that $\Delta(S_1, M) \subset \Delta(S_2, M)$. \textbf{The claim follows}.
	
	We now compute $R_{n,d}$:
	\begin{align*}
		R_{n,d} :=& \inf\left\{ {\rm{Rank}}\Big(A(z,\bar{z})\|z\|^2 \Big) \,\bigg|\, A \in \mathbb{C}_{(d,d)}[z,\bar{z}],\ A \notin {\rm{SOS}}_n,\ J_n(A) \in {\rm{SOS}}_n \right\}\\
		=& \inf\left\{ {\rm{Rank}}\Big(J_n(H) \Big) \,\bigg|\, H \in {\rm{Herm}}_{\binom{n+d-1}{d}}(\mathbb{C}),\ H \ngeq 0,\ J_n(H) \geq 0 \right\}\\
		=& \inf\left\{ {\rm{Rank}}\Big(J_n(H) \Big) \,\bigg|\, H \in {\rm{Cone}}(S_2),\ H \ngeq 0 \right\}.
	\end{align*}
	
	\noindent\textbf{Step 2.}
	For any $0 \neq H \in {\rm{Cone}}(S_2)$, there exists $H' \in {\rm{Cone}}(S_2)$ with ${\rm{Tr}}(H') = 1$ such that
	\begin{equation}
		\label{trace 1}
		{\rm{Rank}}(J_n(H)) = {\rm{Rank}}(J_n(H')).
	\end{equation}
	Actually, when ${\rm{Tr}}(H) > 0$, take $H' = \frac{H}{{\rm{Tr}}(H)}$; when ${\rm{Tr}}(H) = 0$, then $H = 0$. 
	
	Using the homogeneity  (\ref{trace 1}).
	\begin{align*}
		R_{n,d} 
		=& \inf\left\{ {\rm{Rank}}\Big(J_n(H) \Big) \,\bigg|\, H \in {\rm{Cone}}(S_2),\ H \ngeq 0,\ {\rm{Tr}}(H) = 1 \right\}\\
		=& \inf\left\{ {\rm{Rank}}\Big(J_n(H) \Big) \,\bigg|\, H \in \Delta(S_2, M) \setminus \Delta(S_1, M) \right\}.
	\end{align*}
	
	\noindent\textbf{Step 3.}
	Since for any two positive semidefinite Hermitian matrices $H_1, H_2$,  ${\rm{Rank}}(H_1 + H_2) \geq \max\left\{ {\rm{Rank}}(H_1), {\rm{Rank}}(H_2) \right\}$, we have that  for any  $H_1, H_2 \in \Delta(S_2, M) \setminus \Delta(S_1, M)$ with $H_1 \neq H_2$ and $0 < \lambda < 1$, it holds that 
	\begin{align*}
		{\rm{Rank}}\Big(J_n(\lambda H_1 + (1-\lambda)H_2)\Big) &= {\rm{Rank}}\Big(\lambda J_n(H_1) + (1-\lambda)J_n(H_2)\Big) \\
		&\geq \max\left\{ {\rm{Rank}}(J_n(H_1)), {\rm{Rank}}(J_n(H_2)) \right\}.   
	\end{align*}
	Thus, if the set $\Delta(S_2, M) \setminus \Delta(S_1, M)$ has non-empty extreme points, we have
	\[
	R_{n,d} = \inf_{X \in E(\Delta_2\setminus \Delta_1)} {\rm{Rank}}(J_n(X)) = \inf_{X \in E(\Delta_2) \setminus E(\Delta_1)} {\rm{Rank}}(J_n(X)).
	\]
	The lastest equality uses the lemma \ref{extreme point outside}.

	By  \cite[Theorem 18.5]{R1}, any closed convex set that does not contain a line is the closure of its extreme points and extreme directions. Since Newton-Okounkov bodies are compact sets, they do not contain any lines or extreme directions, so $E(\Delta(S_i, M)) \neq \emptyset$, for $i=1,2$. Furthermore, $E(\Delta(S_2, M)) \setminus E(\Delta(S_1, M)) \neq \emptyset$, thus the set $\Delta(S_2, M) \setminus \Delta(S_1, M)$ has non-empty extreme points; otherwise, $\Delta(S_2, M) \subset \Delta(S_1, M)$, which is absurd. 
	
	The proof of Theorem \ref{thm:main-general} is complete.
	\section{Newton-Okounkov body associatied with SOS conjecture in the diagonal case}\label{sect: diagonal}
	In this section, we restrict our discussions to the SOS conjecture in the diagonal case, and complete the proof of Theorem \ref{thm:main-diagonal}. It suffices to consider the real valued,  bihomogeneous of degree $(d,d)$, Hermitian polynomial $A(z,\bar z)$.

	
	For any $n \geq 1$ and $d \geq 0$, set  $N=\binom{n+d-1}{d}$. With respect to the lexicographic order of the holomorphic polynomial of degrees $d$, there is a canonical isomorphism $\mathbb{C}_{(d,d)}^{Diag}[z,\bar{z}] \cong \mathbb{R}^{N}$ regarding the associated Hermitian matrix $H$ as a column vector. 
	
	Let $J_{n,d}$ be the associated prolongation map defined as follows:
	\[
	J_n : \mathbb{C}_{(d,d)}^{Diag}[z,\bar{z}] \to \mathbb{C}_{(d+1,d+1)}^{Diag}[z,\bar{z}],\quad H \mapsto J_{n,d}H.
	\]
	In \cite{WYZ}, we give a Macaulay type matrix representation of $J_{n,d}$ with respect to the lexicographic order of the holomorphic polynomial of degrees $d$ and $d+1$ as follows:
	for any $n \geq 1$ and $d \geq 0$, define $J_{1,d} = 1$ and
	\begin{align}\label{equ: mac repr Jnd}
		J_{n,d} =
		\begin{pmatrix}
			1 &&&&&\\
			J_{n-1,0}& I &&&&\\
			& J_{n-1,1} & I &&&\\
			&&\cdots &&I&\\
			&&&\cdots & J_{n-1,d-1} &I \\
			&&&&&J_{n-1,d}
		\end{pmatrix}.
	\end{align}
	Moreover, $J_{n,d}$ is a $\binom{n+d}{d+1} \times \binom{n+d-1}{d}$ matrix with exactly $n$ entries equal to $1$ in each column and zeros elsewhere. 
	
	
	For all $n,d \geq 2$, let
	\[
	M^{Diag} = {\rm{Tr}}^{-1}(\mathbb{R}_{\geq 0}) \cap \mathbb{C}_{(d,d)}^{Diag}[z,\bar{z}] ,
	\]
	\[
	S_1^{Diag} = \left\{ H \in \mathbb{R}^{\binom{n+d-1}{d}} \mid H \geq 0 \right\}, \quad S_2^{Diag} = \left\{ H \in \mathbb{R}^{\binom{n+d-1}{d}} \mid J_{n,d}H \geq 0 \right\}.
	\]
	
	By the same arguments as in \S \ref{sec: general case}, we obtain
	\begin{align*}
		R_{n,d}^{Diag} :=& \inf\left\{ {\rm{Rank}}\Big(A(z,\bar{z})\|z\|^2 \Big) \,\bigg|\, A \in \mathbb{C}_{(d,d)}^{Diag}[z,\bar{z}],\ A \notin {\rm{SOS}}_n,\ J_{n,d}(A) \in {\rm{SOS}}_n \right\}\\
		=& \inf\left\{ {\rm{Rank}}\Big(J_n(H) \Big) \,\bigg|\, H \in \Delta(S_2^{Diag}, M^{Diag}) \setminus \Delta(S_1^{Diag}, M^{Diag}) \right\}\\
		=& \inf_{X \in E(\Delta_2^{Diag}) \setminus E(\Delta_1^{Diag})} {\rm{Rank}}(J_n(X)).
	\end{align*}
	
	To complete the proof of Theorem \ref{thm:main-diagonal}, it is only left to show that $E(\Delta_2^{Diag})$ contains only finitely many rational points. 
	
	If we generalize the definition of a half-space to allow it to not contain the origin, then a polyhedral convex set $C \subset \mathbb{R}^N$ formed by the intersection of finitely many rational half-spaces can be expressed as the solution set of linear constraints:
	\[
	\langle x, a_i \rangle \geq k_i,\ 1 \leq i \leq m,\quad a_i \in \mathbb{Q}^N,\ k_i \in \mathbb{Q}.
	\]
	
	For any point $p \in C$, let $\mathrm{I}(p)$ denote the set of indices $i$ for which the inequality is tight at $p$ (i.e., $\langle p, a_i \rangle = k_i$). By the definition of extreme points and basic linear algebra, we have the following lemma.
	\begin{lemma}
		A point $p$ is an extreme point of $C$ if and only if ${\rm{Rank}}(A_{{\rm{I}}(p)}) = N$, where $A_{{\rm{I}}(p)}$ is the matrix whose column vectors are  $a_i$, with $i\in {\rm{I}}(p)$.
	\end{lemma}
	
	Since the components of $A_{{\rm{I}}(p)}$ and $k_i$ are all rational, Cramer's rule implies $p \in \mathbb{Q}^N$. The Newton-Okounkov body $\Delta(S_2^{Diag}, M^{Diag})$ is exactly the solution set of the following rational-coefficient inequalities, consisting of $2 + \binom{n+d}{d+1}$ inequalities:
	\begin{equation}\label{eq:7}
		\begin{cases}
			{\rm{Tr}}(H) \geq 1, \\
			-{\rm{Tr}}(H) \geq -1, \\
			J_{n,d}H \geq 0.
		\end{cases}
	\end{equation}
	Thus, the number of extreme points of $E(\Delta_2^{Diag})$ is at most $C(2 + \binom{n+d}{d+1},\binom{n+d-1}{d})$.  Thus the proof of \cref{thm:main-diagonal} is complete.
	
By reducing the verification of the SOS Conjecture to a computationally tractable problem in the diagonal case, \Cref{thm:main-diagonal} provides a new tool for attacking it. The proposed algorithm in \cref{eq:7} can in principle compute all extremal points and thus the minimal rank $R_{n,d}^{\text{Diag}}$. For the present study, we only present a representative subset of results with acceptable computational cost in the following table  and omit exhaustive calculations.
	
	\begin{table}[htbp]
		\centering
		\begin{tabular}{|c|c|c|c|c|c|c|}
			\hline
			\diagbox[width=3cm, height=1.5cm]{$n$}{$R_{n,d}^{Diag}$}{$d$} & 2 & 3 & 4 & 5 & 6 & 7 \\
			\hline
			2 & 2 & 2 & 2 & 2 & 2 &2\\
			\hline
			3 & 5 & 5 & 5 & 5& 5 & 5\\
			\hline
			4 & 8 & 8 & 8 & 8& 8& 8\\
			\hline
			5 & 14 & 14 &  & & & \\
			\hline
			6 & 20 &  &  & & & \\
			\hline
			7 & 27 &  &  &  & &\\
			\hline
			8 & 35 &  &  &  & &\\
			\hline
			9 & 44 &  &  &  & &\\
			\hline
			10 & 54 &  &  &  & &\\
			\hline
			11 & 65 &  &  &  & &\\
			\hline
		\end{tabular}
	\end{table}
	From the above table, it seems to us that 
	\[ R_{n,d}^{Diag}= \binom{n+1}{2}-1 ,\quad n \geq 5.\]
	and $R_{n,d}^{Diag}=R_{n,d+1}^{Diag}$, which implies the  (Weak) SOS conjecture for diagonal case in any dimension using \cite[Theorem 1.1 (B)]{WYZ} and induction argument.

\end{document}